\documentclass{amsart}
\usepackage{mathrsfs,latexsym,amsfonts,amssymb}

\newtheorem{theorem}{Theorem}[section]
\newtheorem{lemma}[theorem]{Lemma}
\newtheorem{corollary}[theorem]{Corollary}
\newtheorem{question}[theorem]{Question}

\theoremstyle{definition}
\newtheorem{definition}[theorem]{Definition}
\newtheorem{proposition}[theorem]{Proposition}
\newtheorem{example}[theorem]{Example}

\begin{document}

\title[On paratopological groups]
{On paratopological groups}

\author{Fucai Lin}
\address{(Fucai Lin): Department of Mathematics and Information Science,
Zhangzhou Normal University, Zhangzhou 363000, P. R. China}
\email{linfucai2008@yahoo.com.cn}

\author{Chuan Liu*}
\address{(Chuan Liu): Department of Mathematics,
Ohio University Zanesville Campus, Zanesville, OH 43701, USA}
\email{liuc1@ohio.edu}
\thanks{The first author is supported by the NSFC (No. 10971185, 10971186) and the Natural Science Foundation of Fujian Province (No. 2011J05013) of China.}

\thanks{*corresponding author}

%insert keywords
%insert subject class

%\date{\today}
\begin{abstract}
In this paper, we firstly construct a Hausdorff non-submetrizable paratopological group $G$ in which every point is a $G_{\delta}$-set, which gives a negative answer to Arhangel'ski\v{\i}\ and Tkachenko's question
[Topological Groups and Related Structures, Atlantis Press and
World Sci., 2008]. We prove that each first-countable Abelian paratopological group is submetrizable. Moreover, we discuss developable paratopological groups and construct a non-metrizable, Moore paratopological group. Further, we prove that a regular, countable, locally $k_{\omega}$-paratopological group is a discrete topological group or contains a closed copy of $S_{\omega}$. Finally, we discuss some properties on non-H-closed paratopological groups, and show that  Sorgenfrey line is not H-closed, which gives a negative answer to Arhangel'ski\v{\i}\ and Tkachenko's question
[Topological Groups and Related Structures, Atlantis Press and
World Sci., 2008]. Some questions are posed.
\end{abstract}

\maketitle

\section{Introduction}
A {\it semitopological group} $G$ is a group $G$ with a topology such
that the product map of $G\times G$ into $G$ is separately
continuous. A {\it paratopological group} $G$ is a group $G$ with a topology such that
the product map of $G \times G$ into $G$ is jointly continuous. If $G$ is a paratopological group  and the inverse operation of $G$ is continuous, then $G$ is called a {\it topological group}.  However, there exists a paratopological group which is not a
topological group; Sorgenfrey line (\cite[Example
1.2.2]{E1989}) is such an example.  Paratopological groups were discussed and many results have been obtained \cite{AR2005, A2008, A2009, CJ, LC, LL2010, LC1, LP}.

\begin{proposition}\label{p0}\cite{R2001}
For a group with topology $(G, \tau)$ the following conditions are equivalent:
\begin{enumerate}
\item $G$ is a paratopological group;

\item The following Pontrjagin conditions for basis $\mathscr{B}=\mathscr{B}_{\tau}$ of the neutral element $e$ of $G$ are satisfied.
\begin{enumerate}
\item ($\forall U, V\in\mathscr{B}$)($\exists W\in\mathscr{B}$): $W\subset U\cap V;$

\item ($\forall U\in\mathscr{B}$)($\exists V\in\mathscr{B}$): $V^{2}\subset U;$

\item ($\forall U\in\mathscr{B}$)($\forall x\in U$)($\exists V\in\mathscr{B}$): $Vx\subset U;$

\item ($\forall U\in\mathscr{B}$)($\forall x\in G$)($\exists V\in\mathscr{B}$): $xVx^{-1}\subset U;$

\smallskip
\end{enumerate}
\end{enumerate}

The paratopological group $G$ is Hausdorff if and only if

(e)  $\bigcap\{UU^{-1}: U\in\mathscr{B}\}=\{e\}$;\\

The paratopological group $G$ is a topological group if and only if

(f)  ($\forall U\in\mathscr{B}$)($\exists V\in\mathscr{B}$): $V^{-1}\subset U.$\\

\end{proposition}

In this paper, we mainly discuss the following questions.

\begin{question}\cite[Open problem 3.3.1]{A2008}\label{q2}
Suppose that $G$ is a Hausdorff (regular) paratopological group in which every point is a $G_{\delta}$-set. Is $G$ submetrizable?
\end{question}

\begin{question}\cite[Open problem 5.7.2]{A2008}\label{q6}
Let $G$ be a regular first-countable $\omega$-narrow paratopological group . Is $G$ submetrizable?
\end{question}

\begin{question}\cite[Problem 20]{AB}\label{q1}
Is every regular first countable (Abelian) paratopological group submetrizable?
\end{question}

\begin{question}\cite[Problem 22]{AB}\label{q7}
Is it true that every regular first countable (Abelian) paratopological group $G$ has a zero-set diagonal\footnote{We say that a space $X$ has a {\it zero-set diagonal} if the diagonal in $X\times X$ is a zero-set
of some continuous real-valued function on $X\times X$.}?
\end{question}

\begin{question}\cite[Problem 21]{AB}\label{q4}
Is every regular first countable (Abelian) paratopological group Dieudonn\'{e} complete?
\end{question}

\begin{question}\cite[Open problem 3.4.3]{A2008}\label{q0}
Let $G$ be a regular $\omega$-narrow first-countable paratopological group. Does there exist a continuous isomorphism of $G$ onto a regular (Hausdorff) second-countable paratopological group?
\end{question}

\begin{question}\label{q9}\cite{LC1}
Is a regular symmetrizable paratopological group metrizable?
\end{question}

\begin{question}\cite[Open problem 5.7.5]{A2008}\label{q3}
Is every paratopological group, which is Moore space, metrizable?
\end{question}

\begin{question}\cite[Open problem 3.6.5]{A2008}\label{q5}
Must the Sorgenfrey line $S$ be closed in every Hausdorff paratopological group containing it as a paratopological subgroup?
\end{question}
\bigskip

We shall give  negative answers to Questions~\ref{q2}, ~\ref{q9}, ~\ref{q5}, and~\ref{q3}, and give a partial answer to Question~\ref{q6}. Moreover, we shall also give affirmative answers to Questions~\ref{q1}, ~\ref{q7}, ~\ref{q0} and ~\ref{q4} when the group $G$ is Abelian.

\section{Preliminaries}
\begin{definition}
Let $\mathscr{P}=\bigcup_{x\in X}\mathscr{P}_{x}$ be a cover of a
space $X$ such that for each $x\in X$, (a) if $U,V\in
\mathscr{P}_{x}$, then $W\subset U\cap V$ for some $W\in
\mathscr{P}_{x}$; (b) the family $\mathscr{P}_{x}$ is a network of $x$ in $X$,
i.e., $x\in\bigcap\mathscr{P}_x$, and if $x\in U$ with $U$ open in
$X$, then $P\subset U$ for some $P\in\mathscr P_x$.

The family $\mathscr{P}$ is called a {\it weak base} for $X$ \cite{Ar} if, for every $A\subset X$, the set $A$ is open in $X$ whenever for each $x\in A$ there exists $P\in
\mathscr{P}_{x}$ such that $P\subset A$.
The space $X$ is {\it weakly first-countable} if $\mathscr{P}_{x}$ is countable for each
$x\in X$.
\end{definition}

\begin{definition}
\begin{enumerate}
\item A space $X$ is called an {\it
$S_{\omega}$}-{space} if $X$ is obtained by identifying all the limit
points from a topological sum of countably many convergent sequences;

\item A space $X$ is called an {\it $S_{2}$}-{space} ({\it Arens' space})  if
$X=\{\infty\}\cup \{x_{n}: n\in \mathbb{N}\}\cup\{x_{n}(m): m, n\in
\mathbb{N}\}$ and the topology is defined as follows: Each
$x_{n}(m)$ is isolated; a basic neighborhood of $x_{n}$ is
$\{x_{n}\}\cup\{x_{n}(m): m>k, \mbox{for some}\ k\in \mathbb{N}\}$;
a basic neighborhood of $\infty$ is $\{\infty\}\cup (\bigcup\{V_{n}:
n>k\ \mbox{for some}\ k\in \mathbb{N}\})$, where $V_{n}$ is a
neighborhood of $x_{n}$.
\end{enumerate}
\end{definition}

\begin{definition}
Let $X$ be a space and $\{\mathscr{P}_{n}\}_{n}$ a sequence of
collections of open subsets of $X$.
\begin{enumerate}
\item $\{\mathscr{P}_{n}\}_{n}$ is called
a {\it quasi-development}~\cite{Be} for $X$ if for every $x\in U$
with $U$ open in $X$, there exists an $n\in \mathbb{N}$ such that
$x\in\mbox{st}(x, \mathscr{P}_{n})\subset U$.

\item $\{\mathscr{P}_{n}\}_{n}$ is called a {\it
development}~\cite{WJ} for $X$ if $\{\mbox{st}(x, \mathscr{P}_{n})\}_{n}$
is a neighborhood base at $x$ in $X$ for each
point $x\in X$.

\item $X$ is called {\it quasi-developable} (resp. {\it developable}), if $X$ has a
quasi-development~(resp. {\it development}).

\item $X$ is called {\it Moore}, if $X$ is regular and developable.
\end{enumerate}
\end{definition}

A subset $B$ of a paratopological group $G$ is called {\it
$\omega$-narrow} in $G$ if, for each neighborhood $U$ of the
neutral element of $G$, there is a countable subset $F$ of $G$ such that
$B\subset FU\cap UF$.

A space $X$ is called a {\it submetrizable} space if it can be mapped onto a
metric space by a continuous one-to-one map. A space $X$ is called a {\it subquasimetrizable} space if it can be mapped onto a
quasimetric space by a one-to-one map.

All spaces are $T_0$ unless stated otherwise. The notations $\mathbb{R, Q, P, N, Z}$ are real numbers, rational numbers, irrational numbers,  natural numbers and integers respectively. The letter $e$
denotes the neutral element of a group. Readers may refer to
\cite{A2008, E1989, Gr} for notations and terminology not
explicitly given here.

\section{Submetrizablity of first-countable paratopological groups}
In this section, we firstly give a negative answer to Question~\ref{q2}, then a answer to Question~\ref{q6}. We also give affirmative answers to Questions~\ref{q1}, ~\ref{q7}, ~\ref{q0} and ~\ref{q4} when $G$ is Abelian.

\begin{proposition}\label{p1}\cite{PN}
The following conditions are equivalent for an arbitrary space $X$.
\begin{enumerate}
\item The space $X$ is submetrizable.

\item The free paratopological group $F_{p}(X)$ is submetrizable.

\item The free Abelian paratopological group $A_{p}(X)$ is submetrizable.
\end{enumerate}
\end{proposition}

\begin{proposition}\label{p2}\cite{PN}
The following conditions are equivalent for an arbitrary space $X$.
\begin{enumerate}
\item The space $X$ is subquasimetrizable.

\item The free paratopological group $F_{p}(X)$ is subquasimetrizable.

\item The free Abelian paratopological group $A_{p}(X)$ is subquasimetrizable.
\end{enumerate}
\end{proposition}

\begin{example}
There exist a Hausdorff paratopological group $G$ in which every point is a $G_{\delta}$-set, and $G$ is not submetrizable.
\end{example}

\begin{proof}
Let $X$ be the lexicographically ordered set $X=(\mathbb{R}\times\{0\})\cup (\mathbb{P}\times\mathbb{Z})$. Then $X$ is a non-metrizable linearly ordered topological space without $G_{\delta}$-diagonal (\cite[Example 2.4]{BH}), hence $X$ is not submetrizable. However, $X$ is quasi-developable \cite{LD}. It is well known that quasi-developability in a generalized ordered space
is equivalent to the existence of a $\sigma$-disjoint base, or of a $\sigma$-point-finite base \cite[Theorem 4.2]{BL2002}. Hence $X$ has $\sigma$-point finite base. Therefore, $X$ is quasi-metrizable since a space with a $\sigma$-point finite base is quasi-metrizable \cite[Page 489]{Gr}. Let $G$ be the free Abelian paratopological group $A_{p}(X)$ over $X$. Since a totally order space endowed the order topology is Tychonoff, then $X$ is Tychonoff. It follows from \cite[Proposition 3.8]{PN} that $G$ is Hausdorff. By Propositions~\ref{p1} and ~\ref{p2}, $G$ is subquasimetrizable and non-submetrizable. Since $G$ is subquasimetrizable, every singleton of $G$ is a $G_{\delta}$-set.
\end{proof}

Next we partially answer Question~\ref{q6}.

The {\it weak extent} \cite{BD} of a space $X$, denoted by
$\omega e(X)$, is the least cardinal number $\kappa$ such that for every open cover $\mathscr{U}$ of $X$ there
is a subset $A$ of $X$ of cardinality no greater than $\kappa$ such that $\mbox{st}(A; \mathscr{U})=X$, where $\mbox{st}(A; \mathscr{U})=\bigcup\{U: U\in\mathscr{U}, U\cap A\neq\emptyset\}$. If $X$ is separable, then $we(X)=\omega$.

\begin{theorem}\cite{BD}\label{l4}
If $X^{2}$ has countable weak extent and a regular $G_{\delta}$-diagonal\footnote{A space $X$ is said to have a {\it regular $G_{\delta}$-diagonal} if the diagonal $\Delta=\{(x, x): x\in X\}$ can be represented as the intersection of the closures of a countable family
of open neighborhoods of $\Delta$ in $X \times X$.}, then
$X$ condenses onto a second countable Hausdorff space.
\end{theorem}

\begin{theorem}\label{l7}\cite{LP}
Each $\omega$-narrow first-countable paratopological group is separable.
\end{theorem}

\begin{theorem}
If $G$ is a regular $\omega$-narrow first-countable paratopological group, then $G$ condenses onto a second countable Hausdorff space..
\end{theorem}

\begin{proof}
 It is straightforward to prove that the product of two $\omega$-narrow paratopological groups is an $\omega$-narrow paratopological group. Then $G^{2}$ is an $\omega$-narrow first-countable paratopological group, and hence $G^{2}$ is separable by Theorem~\ref{l7}. Then $G^{2}$ has countable weak extent. Moreover, it follows from \cite{LC} that $G$ has a regular $G_{\delta}$-diagonal. Therefore, $G$ condenses onto a second countable Hausdorff space by Theorem~\ref{l4}.
\end{proof}

\begin{corollary}
Let $(G, \tau)$ be a regular $\omega$-narrow first-countable paratopological group. There exists a continuous isomorphism of $G$ onto a Hausdorff second-countable space.
\end{corollary}

A paratopological group $G$ with a base at the neutral element $\mathscr{B}$ is a {\it SIN}-{\it group} (Small Invariant Neighborhoods), if for each $U\in \mathscr{B}$ there exists a $V\in\mathscr{B}$ such that $xVx^{-1}\subset U$ for each $x\in G.$

\begin{theorem}\label{t0}
If $(G, \tau)$ is a Hausdorff SIN first-countable paratopological group, then $G$ is submetrizable.
\end{theorem}

\begin{proof}
Let $\{U_{n}: n\in\mathbb{N}\}$ be a countable local base of $(G, \tau)$ at the neutral element $e$, where $U_{n+1}\subset U_{n}$ for each $n\in \mathbb{N}$.

For $x\in G$, let $\mathcal{B}_x=\{xU_nU_n^{-1}: n\in \mathbb{N}\}$. Then $\{\mathcal{B}_x\}_{x\in G}$ has the following properties.

(BP1) For every $x\in G$, $\mathcal{B}_x\neq \emptyset$ and for every $U\in \mathcal{B}_x, x\in U$.

(BP2) If $x\in U\in \mathcal{B}_y$, then there exists a $V\in \mathcal{B}_x$ such that $V\subset U$

In fact, if $x\in U=yU_iU_i^{-1}\in \mathcal {B}_y$, then $x=yu_1u_2^{-1}$ for some $u_1, u_2 \in U_i$. Pick $U_j, U_k\in \{U_{n}: n\in\mathbb{N}\}$ such that $U_k\subset U_i$, $u_1U_k\subset U_i, U_ju_2\subset U_k, u_2^{-1}U_ju_2\subset U_k$. Then $xU_jU_j^{-1}=yu_1u_2^{-1}U_jU_j^{-1}\subset yu_1U_ku_2^{-1}U_j^{-1}\subset yU_i(U_ju_2)^{-1}\subset yU_iU_k^{-1}\subset yU_iU_i^{-1}=U$.

(BP3) For any $V_1, V_2 \in \mathcal{B}_x$ there exists a $V\in \mathcal {B}_x$ such that $V\subset V_1 \cap V_2$.

Let $\tau^{\ast}$ be the topology generated by the neighborhood system $\{\mathcal{B}_x\}_{x\in G}$. Obviously, the topology of $(G, \tau^{\ast})$ is coarser than $(G, \tau)$ and it is first-countable. We prove that $(G, \tau^{\ast})$ is a Hausdorff topological group.

It is easy to see that (a), (d) and (f) in Proposition~\ref{p0} are satisfied. (BP2) implies (c). We check conditions (b) and (d).

 Fix $n\in \mathbb{N}$. Then there is  $k>n$ such that $U_{k}^{2}\subset U_{n}$ since $(G, \tau)$ is a paratopological group. $G$ is a SIN-group, there exists an $m>k$ such that $xU_{m}x^{-1}\subset U_{k}$ for each $x\in G$, and hence $U_{m}^{-1}U_{m}\subset U_{k}U_{m}^{-1}$.  Then $U_{m}U_{m}^{-1}U_{m}U_{m}^{-1}\subset U_{m}U_{k}U_{m}^{-1}U_{k}^{-1}\subset U_{k}U_{k}U_{k}^{-1}U_{k}^{-1}\subset U_{n}U_{n}^{-1}$. Hence (b) is satisfied.

Since $\{U_n: n\in \mathbb{N}\}$ is a network of $G$ at $e$, $\bigcap_{n\in\mathbb{N}} U_nU_{n}^{-1}=\{e\}$ by \cite[Proposition 3.4, 3.5]{A2002}.
 Therefore, $(G, \tau^{\ast})$ is Hausdorff, and hence it is Tychonoff.

Since first-countable topological group metrizable, we have $(G, \tau^{\ast})$ is metrizable. Therefore, $G$ is submetrizable.
\end{proof}

It is well known that all submetrizable spaces have
a zero-set diagonal. Therefore, Theorem~\ref{t0} gives a partial answer to Question~\ref{q7}.

\begin{corollary}\label{c5}
If $(G, \tau)$ is a Hausdorff Abelian first-countable paratopological group, then $G$ is submetrizable.
\end{corollary}

Indeed, we have the following more stronger result.

\begin{theorem}\label{t9}
If $(G, \tau)$ is a Hausdorff Abelian paratopological group with a countable $\pi$-character, then $G$ is submetrizable.
\end{theorem}

\begin{proof}
Let $\mathscr{B}=\{U_{\alpha}: \alpha<\kappa\}$ be a local base at the neutral element $e$.
It follows from the proof of Theorem~\ref{t0} that the family $\{U_{\alpha}U_{\alpha}^{-1}: \alpha<\kappa\}$ is a local base at point $e$ in the Tychonoff topological group $(G, \tau^{\ast})$.

Let $\mathscr{C}=\{V_{n}: n\in\mathbb{N}\}$ be a local $\pi$-base at $e$. Put $\mathscr{F}=\{V_{n}V_{n}^{-1}: n\in \mathbb{N}\}$. Then $\mathscr{F}'=\{int(V_{n}V_{n}^{-1}): n\in \mathbb{N}\}$ is a local base at $e$ in $\tau^{\ast}$.

Indeed, for each $n\in\mathbb{N}$ and fix a point $x\in V_{n}$, then $x^{-1}V_{n}$ is an open neighborhood at $e$ in $\tau$, and hence there exists an $U_{\alpha}\in\mathscr{B}$ such that $U_{\alpha}\subset x^{-1}V_{n}$. Thus $U_{\alpha}U_{\alpha}^{-1}\subset x^{-1}V_{n}V_{n}^{-1}x=V_{n}V_{n}^{-1}$, which implies that $V_{n}V_{n}^{-1}$ is a neighborhood of $e$ in $\tau^{\ast}$. On the other hand, fix $\alpha<\kappa$, there is $n\in \mathbb{N}$ such that $V_n\subset U_\alpha$. $V_nV_n^{-1}\subset U_{\alpha}U_{\alpha}^{-1}$. Therefore, $\mathscr{F}'=\{int(V_{n}V_{n}^{-1}): n\in \mathbb{N}\}$ is a local base at $e$ in $\tau^{\ast}$.

Since first-countable topological group metrizable, we have $(G, \tau^{\ast})$ is metrizable. Therefore, $G$ is submetrizable.
\end{proof}

Since every submetrizable space is (hereditarily) Dieudonn\'{e} complete, Corollary~\ref{c5} give a partial answer to Question~\ref{q4}.

\begin{theorem}\label{c3}
Let $(G, \tau)$ be a Hausdorff separable SIN first-countable paratopological group. There exists a continuous isomorphism of $G$ onto a Tychonoff second-countable topological group.
\end{theorem}

\begin{proof}
By the proof of Theorem~\ref{t0}, we know that $(G, \tau^{\ast})$ is metrizable. Since $G$ is separable and $\tau^{\ast}\subset \tau$, $(G, \tau^{\ast})$ is separable, and hence $(G, \tau^{\ast})$ is a second-countable topological group. Hence there exists a continuous isomorphism of $G$ onto a Tychonoff second-countable topological group.
\end{proof}

By Theorem~\ref{l7}, we have the following corollary, which gives a partial answer to Question~\ref{q0}.

\begin{corollary}
Let $(G, \tau)$ be a Hausdorff $\omega$-narrow first-countable SIN paratopological group. There exists a continuous isomorphism of $G$ onto a Tychonoff second-countable topological group.
\end{corollary}

The following two theorems give another answers to Questions~\ref{q1} and~\ref{q4}.

\begin{theorem}
If $(G, \tau)$ is a Hausdorff saturated\footnote{A paratopological group $G$ is
{\it saturated} if for any neighborhood $U$ of $e$ the set $U^{-1}$ has nonempty
interior in $G$.} first-countable paratopological group, then $G$ is submetrizable.
\end{theorem}

\begin{proof}
Suppose that $\{U_{n}: n\in\mathbb{N}\}$ is a countable local base at $e$. Let $$\sigma=\{U\subset G: \mbox{There exists an}\ n\in \mathbb{N}\ \mbox{such that}\ xU_{n}U_{n}^{-1}\subset U\ \mbox{for each}\ x\in U\}.$$ Since $G$ is saturated, it follows from \cite[Theorem 3.2]{BT} that $(G, \sigma)$ is a topological group. Obvious, $(G, \sigma)$ is $T_{1}$ since $(G, \tau)$ is a Hausdorff, and hence $(G, \sigma)$ is regular. Then $(G, \sigma)$ is first-countable, and thus it is metrizable. Therefore, $(G, \tau)$ is submetrizable.
\end{proof}

\begin{theorem}
If $(G, \tau)$ is a Hausdorff feebly compact\footnote{A space $X$ is called {\it feebly compact} if each locally finite  family of open subsets in $X$ is finite.} first-countable paratopological group, then $G$ is submetrizable.
\end{theorem}

\begin{proof}
Suppose that $\{U_{n}: n\in\mathbb{N}\}$ is a local base at $e$. Then the family $\{\mbox{int}\overline{U_{n}}: n\in\mathbb{N}\}$ is a local base at $e$ for a regular paratopological group topology $\sigma$ on $G$. Obviously, $(G, \sigma)$ is feebly compact. Since $(G, \sigma)$ is first-countable, $(G, \sigma)$ has a regular $G_{\delta}$-diagonal \cite{LC}. Then $(G, \sigma)$ is metrizable since a regular feebly compact space with a regular $G_{\delta}$-diagonal is metrizable. Therefore, $(G, \tau)$ is submetrizable.
\end{proof}

Next, we gives an another answer to Question~\ref{q2}.

\begin{theorem}\label{t5}
Let $(G, \tau)$ be a Hausdorff SIN  paratopological group. If $G$ is locally countable, then $G$ is submetrizable.
\end{theorem}

\begin{proof}
Since $G$ is locally countable, there exists an open neighborhood $U$ of $e$ such that $U$ is a countable set. Then $UU^{-1}$ is also a countable set. Let $UU^{-1}\setminus\{e\}=\{x_{n}: n\in\mathbb{N}\}$. Since $G$ is a paratopological group, we can find a family of countably many neighborhoods $\{V_{n}: n\in\omega\}$ of $e$ satisfying the following conditions: \\
(i) $V_{0}=U$;\\
(ii) For each $n\in\omega$, then $V_{n+1}^{2}\subset V_{n}$;\\
(iii) For each $n\in\mathbb{N}$, then $x_{n}\not\in V_{n}V_{n}^{-1}$ (This is possible since $G$ is Hausdorff.)\\
(iv) For each $n\in\omega$, then $xV_{n+1}x^{-1}\subset V_{n}$ for each $x\in G$.

Since $G$ is a SIN group, the topology $\sigma$ generated by the neighborhood basis $\{V_{n}: n\in\mathbb{N}\}$ is a paratopological group. Clearly, ($G$, $\sigma$) is coarser than $\tau$. Since $\bigcap_{n\in\mathbb{N}}V_{n}V_{n}^{-1}=\{e\}$ by (iii), $\sigma$ is Hausdorff. Therefore, $(G, \sigma)$ is a Hausdorff SIN first-countable paratopological group, and it follows from Theorem~\ref{t0} that $G$ is submetrizable.
\end{proof}

\begin{corollary}\label{c6}
Let $(G, \tau)$ be a Hausdorff Abelian  paratopological group. If $G$ is locally countable, then $G$ is submetrizable.
\end{corollary}

\bigskip

\section{Developable paratopological groups}
Recall that a topological space is  {\it symmetrizable} if its topology is generated
by a symmetric, that is, by a distance function satisfying all the usual restrictions
on a metric, except for the triangle inequality \cite{Ar}.

Now, we give a negative answer\footnote{Li, Mou and Wang \cite{LP} also obtained a non-metrizable Moore paratopological group.} to Question~\ref{q9} by modifying \cite[Example 2.1]{LL2010}.

\begin{example}
There exists a separable, Moore paratopological group $G$ such that $G$ is not metrizable.
\end{example}

\begin{proof}
Let $G=\mathbb{R}\times \mathbb{Q}$ be the group with the usual addition. Then we define a topology on $G$ by giving a
local base at the neutral element $(0, 0)$ in $G$. For each $n\in \mathbb{N}$, let
$$U_{n}(0, 0)=\{(0, 0)\}\cup\{(x, y): y\geq nx, y< \frac{1}{n}, y\in \mathbb{Q}, x\geq 0\}.$$ Let $\sigma$ be the topology generated by the local base $\{U_{n}: n\in \mathbb{N}\}$ at the neutral element $(0, 0)$. It is easy to see that $(G, \sigma)$ is a semitopological group. Now, we prove that is a paratopological group. Since $G$ is Abelian, it only need to prove that for each $n\in \mathbb{N}$ there exists an $m\in \mathbb{N}$ such that $U_{m}^{2}\subset U_{n}$. Indeed, fix an $n\in \mathbb{N}$, Then we have $U_{4n}^{2}\subset U_{n}$. For two points $(x_{i}, y_{i})\in U_{4n} (i=1, 2)$, where $x_{i}\geq 0, y_{i}<\frac{1}{4n}, y_{i}\geq 4nx_{i}, y_{i}\in\mathbb{Q} (i=1, 2)$. Let $$(x, y)=(x_{1}, y_{1})+(x_{2}, y_{2})=(x_{1}+x_{2}, y_{1}+y_{2}).$$ Obviously, we have $$x=x_{1}+x_{2}\geq 0, y=y_{1}+y_{2}<\frac{1}{4n}+\frac{1}{4n}=\frac{1}{2n}<\frac{1}{n},$$ and $$y=y_{1}+y_{2}\geq 4nx_{1}+4nx_{2}=4nx\geq nx.$$ Then $(x, y)\in U_{n}$, and hence $U_{4n}^{2}\subset U_{n}$. Moreover, it is easy to see that $G$ is regular, separable and first-countable space.

For each $q\in\mathbb{Q}$, it is easy to see that the family $\{\{(x, q)\}: x\in \mathbb{R}\}$ is uncountable discrete and closed, hence $G$ is a $\sigma$-space, and thus it is a $\beta$-space\footnote{Let $(X, \tau)$ be a topological space. A function $g: \omega\times X \rightarrow\tau$ satisfies that
$x\in g(n, x)$ for each $x\in X, n \in\omega$. A space $X$ is a $\beta$-space \cite{Gr} if there is a function $g: \omega\times X \rightarrow\tau$
such that if $x\in g(n, x_{n})$ for each $n\in\omega$, then the sequence $\{x_{n}\}$ has a cluster point in $X$. }. $G$ is a Moore space by \cite[Corollary 2.1]{LL2010}.
 Hence $G$ is semi-metrizable \cite{Gr}. Therefore $G$ is symmetrizable since a space is semi-metrizable if only if it is first-countable and symmetrizable \cite{Gr}. However,  $G$ is not metrizable since $G$ is separable and contains a uncountable discrete closed subset.
\end{proof}

\begin{question}\label{q3}
Is every quasi-developable paratopological (semitopological) group a $\beta$-space?
\end{question}

Next, we give a partial answer to Question~\ref{q3}.

\begin{lemma}\cite[Lemma 1.2]{AR2005}\label{l1}
Suppose that $G$ is a paratopological group and not a topological group. Then
there exists an open neighborhood $U$ of the neutral element $e$ of $G$ such that $U\cap U^{-1}$ is
nowhere dense in $G$, that is, the interior of the closure of $U\cap U^{-1}$ is empty.
\end{lemma}

\begin{theorem}\label{t7}
A regular Baire\footnote{Recall that a space is {\it Baire} if the intersection of a
sequence of open and dense subsets is dense.} quasi-developable paratopological group $G$ is a metrizable topological group.
\end{theorem}

\begin{proof}
Claim: Let $U$ be an arbitrary open neighborhood of $e$. Then $\overline{U^{-1}}$ is also a neighborhood of $e$.

Suppose that $\{\mathscr{U}_{n}: n\in\mathbb{N}\}$ is a family of open subsets of $G$ such that for each $x\in G$ and $x\in V$ with $V$ open in $G$, there exists an $ n\in\mathbb{N}$ such that $x\in\mbox{st}(x, \mathscr{U}_{n})\subset V$. For each $\mathbb{N}$, put
$$A_{n}=\{x\in G: \mbox{st}(x, \mathscr{U}_{n})\subset x\cdot U\}.$$ It is easy to see that $G=\bigcup\{A_{n}: n\in\mathbb{N}\}$. Since $G$ is Baire, there exists  $n_{0}\in\mathbb{N}$ such that $\mbox{int}\overline{A_{n_{0}}}\neq\emptyset$. Therefore, there exist a point $x_{0}\in G$ and $n_{1}\in \mathbb{N}$ such that $\mbox{st}(x_{0}, \mathscr{U}_{n_{1}})\subset \overline{A_{n_{0}}}$. Let $\mathscr{V}=\{U_{1}\cap U_{2}: U_{1}\in\mathscr{U}_{n_{0}}, U_{2}\in\mathscr{U}_{n_{1}}\}$. Put $W=\mbox{st}(x_{0}, \mathscr{V})$.
For each $W\cap A_{n_{0}}$, it is easy to see  $$x_{0}\in \mbox{st}(y, \mathscr{V})\subset \mbox{st}(y, \mathscr{U}_{n_{0}})\subset y\cdot U,$$hence $y^{-1}x_{0}\in U$, so $x_{0}^{-1}y\in U^{-1}$, hence $x_{0}^{-1}\cdot (W\cap A_{n_{0}})\subset U^{-1}$. Moreover, since $W\subset \overline{A_{n_{0}}}$, then $W\subset \overline{W\cap A_{n_{0}}}$. Therefore, we have $$e\in x_{0}^{-1}W\subset x_{0}^{-1}\cdot\overline{W\cap A_{n_{0}}}\subset \overline{x_{0}^{-1}\cdot (W\cap A_{n_{0}})}\subset \overline{U^{-1}}.$$Since $x_{0}^{-1}W$ is an open neighborhood of $e$, the set $\overline{U^{-1}}$ is also a neighborhood of $e$.

$\overline{U^{-1}}\cap U\subset \overline{U\cap U^{-1}}$, by Claim, $\overline{U\cap U^{-1}}$ contains a neighborhood of $e$, then $U\cap U^{-1}$ is not nowhere dense.  By Lemma~\ref{l1}, $G$ is a topological group. Therefore $G$ is metrizable since first-countable topological groups are metrizable.
\end{proof}

Finally, we pose some questions about developable paratopological groups.

\begin{question}
Is each regular Baire quasi-developable semitopological group $G$ is a paratopological group?
\end{question}

If $G$ is developable, then the answer is affirmative. Indeed, it was  proved that each Baire Moore semitopological group $G$ is a metrizable topological group \cite{CJ}.

\begin{question}
Is every developable or Moore paratopological group submetrizable?
\end{question}

\begin{question}
Is every normal Moore paratopological group submetrizable?
\end{question}

\begin{question}
Is every paratopological group with a base of countable order\footnote{A space $X$ is said to have a
{\it base of countable order}(BCO) \cite{Gr} if there is a sequence
$\{\mathcal {B}_{n}\}$ of base for $X$ such that whenever $x\in
b_{n}\in\mathcal {B}_{n}$ and $(b_{n})$ is decreasing (by set
inclusion), then $\{b_{n}: n\in \mathbb{N}\}$ is a base at $x$.} developable?
\end{question}

\bigskip

\section{Fr\'{e}chet-Urysohn paratopological groups}

First, we need the following Lemma.

\begin{lemma}\cite[Theorem 4.7.5]{A2008}\label{l6}
Every weakly first-countable Hausdorff paratopological group is first-countable.
\end{lemma}

Arhangel'skii proved that if a topological group G is an image of a separable metrizable space under a pseudo-open\footnote{A map $f: X\to Y$ is pseudo-open if for each $y\in Y$ and every open set $U$ containing $f^{-1}(y)\subset U$  one has $y\in int(f(U))$.} map, then G is metrizable \cite{A2011}. We have the following.

\begin{theorem}\label{t8}
Let $G$ be a uncountable paratopological group. Suppose that $G$ is a pseudoopen image of a separable metric space, then $G$ is a separable and metrizable.
\end{theorem}

\begin{proof}
We introduce a new product operation in the topological space $G$ by the formula: $a\times b=ba$, for $a, b\in G$ and denote the space with this operation by $H$. Put $T=\{(g, g^{-1})\in G\times H, g\in G\}$. $|T| > \omega$ since $G$ is uncountable. By \cite[Proposition 2.9]{AR2005}, $H$ is a paratopological group and $T$ is closed in the space $G\times H$ and is a topological group.

Since $G$ is a pseusdo-open image of a separable metric space, then the space $G$ is a Fr\'echet-Urysohn space with a countable k-network. $G\times H$ has a countable network. Hence $T$ has a countable network. By the proof of \cite[Theorem 4.9]{Gr}, $T$ is a one-to-one continuous image of a separable metric space $M$.  Let $D$ be a countable dense subset of $M$, there is a sequence $L\subset D$ converging to some point in $M\setminus D$ since $M$ is uncountable. Therefore, there is a non-trivial sequence $\{(g_n, g_n^{-1}): n\in \mathbb{N}\}$ converging to $(e, e)$ (note that $T$ is homogeneous), and hence there exists a sequence $C_0=\{g_n: n\in \mathbb{N}\} \subset G$ converging to $e$ and its inverse $C_1=\{g_n^{-1}: n\in \mathbb{N}\}$ also converges to $e$.
$G$ contains no closed copy of $S_2$ since $G$ is Fr\'echet-Urysohn. By \cite[Theorem 2.4]{LC}, $G$ contains no closed copy of $S_\omega$. Since $G$ is a sequential space with a point-countable k-network and contains no closed copy of $S_\omega$, then $G$ is weakly first-countable \cite{Ls7}, and hence $G$ is first-countable by Lemma ~\ref{l6}. Therefore $G$ is separable and metrizable \cite[Theorem 11.4(ii)]{Gr}.
\end{proof}

A quotient image of a topological sum of
countably many compact spaces is called a $k_{\omega}$-space. Every countable $k_{\omega}$-space is a
sequential $\aleph_{0}$-space, and a product of two $k_{\omega}$-spaces is itself a $k_{\omega}$-space, see \cite{MA}.

\begin{theorem}\label{t6}
Let $G$ be a regular countable, locally $k_{\omega}$, paratopological group. Then $G$ is a discrete topological group or contains a closed copy of $S_{\omega}$.
\end{theorem}

\begin{proof}
Suppose that $G$ is an $\alpha_{4}$-space. Then since $G$ is locally $k_{\omega}$, $G$ is sequential and $\aleph_{0}$, and thus $G$ is weakly first-countable \cite{Ls3}. Then $G$ is first-countable by Lemma~\ref{l6}, and hence $G$ is a separable metrizable space since $G$ is countable. If $G$ is not discrete, $G$ has no isolated points. Then $G$ is homeomorphic to the rational  number set $\mathbb{Q}$ since a separable metrizable space is homeomorphic to the rational number set $\mathbb{Q}$ provided that it is infinite, countable and without any isolated points \cite{HP}. However, $\mathbb{Q}$ is not a locally $k_{\omega}$-space, which is a contradiction. Then $G$ is discrete, and $G$ is  a topological group.

If $G$ is not an $\alpha_{4}$-space, and thus $G$ contains a copy of $S_{\omega}$. Since every point of $G$ is a $G_{\delta}$-set, it follows from \cite[Corollary 3.4]{Ls7} that $G$ contains a closed copy of $S_{\omega}$.
\end{proof}

By Theorem~\ref{t6}, it is easy to obtain the following corollary.

\begin{corollary}
Let $G$ be a regular, countable, non-discrete, Fr\'{e}chet-Urysohn\footnote{A space $X$ is said to be {\it Fr$\acute{e}$chet-Urysohn} if, for
each $x\in \overline{A}\subset X$, there exists a sequence
$\{x_{n}\}$ such that $\{x_{n}\}$ converges to $x$ and $\{x_{n}:
n\in\mathbb{N}\}\subset A$.} paratopological group. If $G$ is $k_{\omega}$, then $G$ contains a closed copy of $S_{\omega}$ and no closed copy of $S_{2}$.
\end{corollary}

Since the closed image of a locally compact, separable metric space is a Fr\'{e}chet-Urysohn and  $k_{\omega}$-space, we have the following corollary.

\begin{corollary}
Let $G$ be a countable non-discrete paratopological group. If $G$ is a closed image of a locally compact, separable metric space, then $G$ contains a closed copy of $S_{\omega}$, and hence $G$ is not metrizable.
\end{corollary}

The condition ``locally $k_{\omega}$-space'' is essential in Theorem~\ref{t6}, and we can not omit it.

\begin{example}\label{e0}
There exists a regular non-discrete countable second-countable paratopological group $G$ such that $G$ contains no closed copy of $S_{\omega}$ and $G$ is not a topological group.
\end{example}

\begin{proof}
Let $G=\mathbb{Q}$ be the rational number and endow with the subspace topology of the Sorgenfrey line. Then $\mathbb{Q}$ is a second-countable paratopological group and non-discrete. Obviously, $G$ contains no closed copy of $S_{\omega}$ and $G$ is not a topological group.
\end{proof}
\bigskip

\section{non-H-closed paratopological groups}
A paratopological group is H-closed if it is closed in every Hausdorff paratopological
group containing it as a subgroup.

Let $U$ be a neighborhood of $e$ in a paratopological group $G$. We say that a
subset $A\subset G$ is $U$-unbounded if $A\nsubseteq KU$ for every finite subset $K\subset G$.

Now, we give a negative answer to Question~\ref{q5}.

\begin{lemma}\cite{RO}\label{l5}
Let $G$ be an abelian paratopological group of the infinite exponent.
If there exists a neighborhood $U$ of the neutral element such that a group $nG$ is $UU^{-1}$-unbounded for
every $n\in\mathbb{N}$, then the paratopological group $G$ is not H-closed.
\end{lemma}

\begin{theorem}\label{t1}
The Sorgenfrey line $(\mathbb{R}, \tau)$ is not H-closed.
\end{theorem}

\begin{proof}
Obvious, $\mathbb{R}$ is an abelian paratopological group of the infinite exponent. Let $U=[0, 1)$. Then $UU^{-1}=(-1, 1)$ For each $n\in\mathbb{N}$, it is easy to see that $n\mathbb{R}\nsubseteq K(-1, 1)$ for every finite subset $K$ of $\mathbb{R}$. Indeed, for each finite subset $K$ of $\mathbb{R}$, since $K$ is finite set, there exists an $n\in \mathbb{N}$ such that $|x|\leq n$ for each $x\in K$, and then $K\subset (-1, 1)\subset (-n, n)$. Therefore, we have $n\mathbb{R}\nsubseteq K(-1, 1)$. By Lemma~\ref{l5}, $\mathbb{R}$ is not H-closed.
\end{proof}

However, we have the following theorem.

\begin{theorem}\label{t4}
Let $(\mathbb{R}, \tau)$ be the Sorgenfrey line. Then the quotient group $(\mathbb{R}/\mathbb{Z}, \xi)$ is H-closed, where $\xi$ is the quotient topology.
\end{theorem}

\begin{proof}
Let $(\mathbb{R}/\mathbb{Z}, \sigma)$ be the finest group topology such that $\sigma\subset \xi$. Then $(\mathbb{R}/\mathbb{Z}, \sigma)$ is compact. Therefore, $(\mathbb{R}/\mathbb{Z}, \sigma)$ is H-closed. Hence $(\mathbb{R}/\mathbb{Z}, \xi)$ is H-closed by \cite[Proposition 10]{RO}.
\end{proof}

Let $G$ be an abelian non-periodic  paratopological group. We say that $G$ is {\it strongly unbounded} if there exists non-periodic element $x_{0}$ and open neighborhood $U$ of $e$ such that $\langle x_{0}\rangle\cap UU^{-1}=\{e\}$ and $\langle x_{0}\rangle$ is closed in $G$, where $\langle x_{0}\rangle$ is a subgroup generated by $x_0$. Obviously, every strongly unbounded paratopological group is not H-closed.

Next, we discuss some non-H-closed paratopological groups.

Given any elements $a_{0}, a_{1},\cdots, a_{n}$ of an abelian group $G$ put
$$X(a_{0}, a_{1},\cdots, a_{n})=\{a_{0}^{x_{0}}a_{1}^{x_{1}}\cdots a_{n}^{x_{n}}: 0\leq x_{i}\leq n, 0\leq i\leq n\}.$$

\begin{theorem}\label{t3}
Let $(G, \tau)$ be a Hausdorff strongly unbounded paratopological group. Then there exists a Tychonoff paratopological group topology $\gamma$ on $G\times \mathbb{Z}$ satisfies the following conditions:
\begin{enumerate}
\item There exists a Hausdorff paratopological group topology $\sigma\subset \gamma$ on $G\times \mathbb{Z}$  such that $\sigma|_{G\times\{0\}}=\tau$ and $(G\times \mathbb{Z}, \sigma)$ contains closed copies of $S_{2}$ and $S_{\omega}$;

\item $(G\times \mathbb{Z}, \gamma)$ is a  strongly zero-dimensional, paracompact $\sigma$-space;

\item The remainder $b(G\times \mathbb{Z}, \gamma)\setminus (G\times \mathbb{Z}, \gamma)$ of every Hausdorff compactification $b(G\times \mathbb{Z}, \gamma)$ is pseudocompact.
 \end{enumerate}
\end{theorem}

\begin{proof}
Since $G$ is strongly unbounded paratopological group, there exist a non-periodic element $x_{0}$ of $G$ and open neighborhood $U$ of $e$ such that $\langle x_{0}\rangle\cap UU^{-1}=\{e\}$ and $\langle x_{0}\rangle$ is closed in $G$. Obvious, $\langle x_{0}\rangle$ is an abelian group. Then the mapping $f: (nx_{0}, m)\mapsto (n, m)$ is is naturally isomorphic from $\langle x_{0}\rangle\times \mathbb{Z}$
 onto a group $\mathbb{Z}\times \mathbb{Z}$. We may assume that $\mathbb{Z}\times \mathbb{Z}\subset G\times \mathbb{Z}$. Now we define a zero dimensional paratopological group topology on $G\times \mathbb{Z}$.
Obvious, we can define a positively natural number sequence $\{a_{n}\}$ satisfies the following conditions: \\
(1) $a_{n}>n$;\\
(2) $a_{n}> 2a$ for each $a\in X(a_{1}, \cdots, a_{n-1}).$\\
Define a base $\mathscr{B}_{\gamma}$ at the neutral element of paratopological group topology $\gamma$ on the group $G\times\mathbb{Z}$ as
follows. Put $A^{+}_{n}=\{(e, 0)\}\cup \{(a_{k}, 1): k>n\}$. For every strictly increasing sequence $\{n_{k}\}$
put $A[n_{k}]=\bigcup_{l\in\mathbb{N}}A^{+}_{n_{1}}\cdots A^{+}_{n_{l}}$. Put $\mathscr{B}_{\gamma}=\{A[n_{k}]\}$. Then $\gamma$ is a zero dimensional paratopological group topology on $G\times \mathbb{Z}$,  see \cite[Lemma 3]{RO}.

(1) By the proof of \cite[Lemma 3]{RO}, we can define a topology $\sigma$ on $G\times \mathbb{Z}$ such that $\sigma|_{G\times\{0\}}=\tau$, $\sigma\subset \gamma$ and $\sigma|_{\mathbb{Z}\times\mathbb{Z}}=\gamma|_{\mathbb{Z}\times\mathbb{Z}}$. Let $\gamma|_{\mathbb{Z}\times\mathbb{Z}}=\xi$.
Since $(\mathbb{Z}\times \mathbb{Z}, \xi)$ is zero dimensional and countable, the space $(\mathbb{Z}\times \mathbb{Z}, \xi)$ is Tychonoff.

{\bf Claim 1}: $(\mathbb{Z}\times \mathbb{Z}, \xi)$ is a closed in $(G\times \mathbb{Z}, \sigma)$.

Since $\langle 1\rangle\cap UU^{-1}=\{e\}$ and $\langle 1\rangle$ is closed in $G$, it is easy to see that $(\mathbb{Z}\times \mathbb{Z}, \xi)$ is closed in $(G\times \mathbb{Z}, \sigma)$.

{\bf Claim 2}:  $(\mathbb{Z}\times \mathbb{Z}, \xi)$ contains a closed copy of $S_{\omega}$.

For each $n\in \mathbb{N}$, let $\beta_{n}=\{(na_{n+k}, n)\}_{k\in\mathbb{N}}$. Obvious, each $\beta_{n}$ converges to $(0, 0)$ as $k\rightarrow \infty$.

Let $X=\{(na_{n+k}, n): k, n\in \mathbb{N}\}\cup\{(0, 0)\}$. It is easy to see that $X$ is a closed copy of $S_{\omega}$.

{\bf Claim 3}:  $(\mathbb{Z}\times \mathbb{Z}, \xi)$ contains a closed copy of $S_{2}$.

Let $\alpha_{0}=\{(a_{k}, 1)\}_{k\in\mathbb{N}}$. For each $n\in \mathbb{N}$, let $\alpha_{n}=\{(a_{n}+(n-1)a_{k}, n)\}_{k\in\mathbb{N}}$. Obvious, $\alpha_{0}$ converges to $(0, 0)$ as $k\rightarrow \infty$ and each $\alpha_{n}$ converges to $(a_{n}, 1)$ as $k\rightarrow \infty$.

Let $X=\{(a_{n}+(n-1)a_{k}, n): k, n\in \mathbb{N}\}\cup\alpha_{0} \cup\{(0, 0)\}$. It is easy to see that $X$ is a closed copy of $S_{2}$.

(2) Let $i: \mathbb{Z\rightarrow \mathbb{}Z}$ be the identity map. Since $\mathbb{Q}$ is a divisible group, the map $i$ can be extended to a homomorphism $\phi: G\rightarrow \mathbb{Q}$. Put $|x|=|\phi(x)|$ for every
element $x\in G$. Then $|\cdot|$ is a seminorm on the group $G$ such that for all $x, y\in G$ holds $|x+y|\leq |x|+|y|$.
For each $n\in\mathbb{N}$, let $\mathscr{B}_{n}=\{\{(x, y)\}: \phi (x)\leq n, y\in\mathbb{Z}\}$. Then it is easy to see that each $\mathscr{B}_{n}$ is a discrete family of closed subsets. Since it follows from Theorem~\ref{t5} that the space $(G\times\mathbb{Z}, \gamma)$ is submetrizable, there exists a metrizale topology $\mathscr{F}$ such that $\mathscr{F}\subset \gamma$. Obvious, $\bigcup_{n\in\mathbb{N}}\mathscr{B}_{n}$ is a network for $(G\times \mathbb{Z}, \mathscr{F})$, and it follows from \cite[Theorem 7.6.6]{A2008} that $(G\times\mathbb{Z}, \gamma)$ is a paracompact $\sigma$-space. Moreover, each subspace $F_{n}=\bigcup\mathscr{B}_{n}$ is strongly zero-dimensional and $G\times\mathbb{Z}=\bigcup_{n\in\mathbb{N}}F_{n}$, and hence $(G\times\mathbb{Z}, \gamma)$ is strongly zero-dimensional by \cite[Theorem 2.2.7]{E1978}.

(3) {\bf Claim 4}: $(G\times \mathbb{Z}, \gamma)$ has no strong $\pi$-base\footnote{A {\it strong $\pi$-base} of a space $X$ at a subset $F$ of $X$ is an infinite family $\gamma$ of non-empty open subsets of $X$ such that
every open neighborhood of $F$ contains all but finitely many elements of $\gamma$. Clearly, a strong $\pi$-base can be always
assumed to be countable.} at any compact subset of $G\times \mathbb{Z}$.

Indeed, suppose there exists a compact subset $K\subset G\times \mathbb{Z}$ such that $K$ has a strong $\pi$-base $\varphi$. It is easy to see that the set $\{n: K\cap (G\times\{n\})\neq\emptyset, n\in \mathbb{Z}\}$ is finite. Moreover, without loss of generalization, we may assume that strong $\pi$-base $\varphi$ at $K$ is countable and each element of $\varphi$ is an open neighborhood of some point in $G\times \mathbb{Z}$. Let $A=\{(b_{i}, n(i)): i\in\mathbb{N}, n(i)\in\mathbb{Z}\}$, and let $\varphi=\{(b_{i}, n(i))+A[n_{k}^{i}]: i\in \mathbb{N}\}$.

Case 1: The set $\{n(i): i\in\mathbb{N}\}$ is infinite.

Then it is easy to see that $A$ contains a discrete closed subset $B$ such that $B\cap K=\emptyset$. Let $U$ be an open neighborhood at $K$. Then $U\setminus B$ is also  an open neighborhood at $K$, and each $(b_{i}, n(i))+A[n_{k}^{i}]\nsubseteq U\setminus B$, which is a contradiction.

Case 2: The set $\{n(i): i\in\mathbb{N}\}$ is finite.

Let $n_{0}=\mbox{max}(\{n(i): i\in\mathbb{N}\}\bigcup\{n: K\cap (G\times\{n\})\neq\emptyset, n\in \mathbb{Z}\}\bigcup\{0\})+1$. For each $i+n_{0}\in\mathbb{N}$, we can choose a $(c_{_{i}}, i+n_{0})\in (b_{i}, n(i))+A[n_{k}^{i}]$. Then $\{(c_{_{i}}, i+n_{0}): i\in\mathbb{N}\}$ is closed and discrete subset in $(G\times \mathbb{Z}, \gamma)$ and $\{(c_{_{i}}, i+n_{0}): i\in\mathbb{N}\}\cap K=\emptyset$. Let $U$ be an open neighborhood at $K$. Then $U\setminus \{(c_{_{i}}, i+n_{0}): i\in\mathbb{N}\}$ is also  an open neighborhood at $K$, and each $(b_{i}, n(i))+A[n_{k}^{i}]\nsubseteq U\setminus \{(c_{_{i}}, i+n_{0}): i\in\mathbb{N}\}$, which is a contradiction.

It follows from Claim 4 and \cite[Corollary 4.3]{A2009} that $b(G\times \mathbb{Z})\setminus (G\times \mathbb{Z})$ is pseudocompact.
\end{proof}

{\bf Remark} (1)  Sorgenfrey line $(\mathbb{R}, \tau)$ is a Hausdorff strongly unbounded paratopological group, and hence we can define a Hausdorff paratopological group topology $\sigma$ on $G\times \mathbb{Z}$  such that $\sigma|_{G\times\{0\}}=\tau$ and $(G\times \mathbb{Z}, \sigma)$ contains closed copies of $S_{2}$ and $S_{\omega}$;

(2) Let $\mathbb{Q}$ be the rationals with the subspace topology of usual topology $\mathbb{R}$. Then $\mathbb{Q}$ is a topological group. It easily check that $\mathbb{Q}$ is a strongly zero-dimensional, nowhere locally compact, paracompact $\sigma$-space. In additional, $\mathbb{Q}$ has a strong $\pi$-base at each point since it is first-countable. It follows from \cite[Lemma 2.1]{A3} that the remainder of any Hausdorff compactification of $\mathbb{Q}$ is not pseudocompact.

A paratopological group $G$ is said to have the property ($^{\star\star}$), if there exists a
sequence $\{x_{n}: n\in\mathbb{N}\}$ of $G$ such that $x_{n}\rightarrow e$ and $x_{n}^{-1}
\rightarrow e$.

In \cite{LC}, C. Liu proved the following theorem.

\begin{theorem}\cite{LC}\label{t2}
Let $G$ be a paratopological group having the property ($^{\star\star}$). Then
$G$ has a (closed) copy of $S_{2}$ if it has a (closed) copy of $S_{\omega}$.
\end{theorem}

It is natural to ask the following.

\begin{question}\cite{LC}\label{q10}
Can we omit the property ($^{\star\star}$) in Theorem~\ref{t2}.
\end{question}

By Lemma~\ref{l5}, the space $G$ in Theorem~\ref{t3} is not H-closed. Moreover, it is easy to see that the paratopological group topology on $\mathbb{Z\times\mathbb{Z}}$ in Theorem~\ref{t3} does not have the the property ($^{\star\star}$). Then we have the following questions.

\begin{question}
Let $(G, \tau)$ be a H-closed paratopological group. Does there exist a Hausdorff paratopological group topology $\sigma$ on $G\times \mathbb{Z}$  such that $\sigma|_{G\times\{0\}}=\tau$ and $(G\times \mathbb{Z}, \sigma)$ contains closed copies of $S_{2}$ and $S_{\omega}$?
\end{question}

\begin{question}
Let $G$ be a not H-closed paratopological group. Is it  true that
$G$ has a (closed) copy of $S_{2}$ if it has a (closed) copy of $S_{\omega}$?
\end{question}

\begin{question}
Let $G$ be a H-closed paratopological group. Is it  true that
$G$ has a (closed) copy of $S_{2}$ if it has a (closed) copy of $S_{\omega}$?
\end{question}

{\bf Acknowledgements}. We wish to thank
the reviewers for the detailed list of corrections, suggestions to the paper, and all her/his efforts
in order to improve the paper.


\begin{thebibliography}{99}
\bibitem{Ar} A.V. Arhangel'ski\v\i, Mappings and spaces,
{\it Russian Math. Surveys}, 1966, {\bf 21}, 115--162.

\bibitem{A2002}
    A.V. Arhangel'ski\v\i, {\it Topological invariants in algebraic enviroment}, Recent progress in General Topology II, ed. by Hu\v{s}ek and van Mill, ELSEVIER, (2002), 1--58.

\bibitem{A3} A. Arhangel'ski\v{\i}, {\it Two types of remainders of topological groups},
Comment. Math. Univ. Carolin., 49(2008), 119--126.

\bibitem{A2011} A.V. Arhangel'ski\v\i, {\it Some classes of quotient and pseudoopen mappings and some new cardinal invariants of tightness type}, International Conference on Topology and its Applications, Islamabad, Pakistan, July 4--10, 2011.

\bibitem{AB} A.V. Arhangel'ski\v{\i}, D.K. Burke, {\it Spaces with a regular $G_{\delta}$-diagonal}, Topology Appl.,
{\bf 153}(2006), 1917--1929.


\bibitem{A2009} A.V. Arhangel'ski\v\i, M.M. Choban,  {\it Remainders of rectifiable spaces}, Topology Appl.,
{\bf 157}(2010), 789--799.

\bibitem{AR2005} A.V. Arhangel'ski\v\i, E.A. Reznichenko,  {\it Paratopological and semitopological groups versus topological groups}, Topology Appl., {\bf 151}(2005), 107--119.


\bibitem{A2008} A.V. Arhangel'ski\v{\i}, M. Tkachenko,
{\it Topological Groups and Related Structures}, Atlantis Press and
World Sci., 2008.



\bibitem{BD} D. Basile, A. Bella, G.J. Ridderbos, {\it Weak extent, submetrizability and diagonal degrees}, arXiv:1112.0883v1[math.GN].

\bibitem{Be} H.R. Bennett, {\it On quasi-developable spaces},
General Topology Appl., {\bf 1}(1971), 253--262.

\bibitem{BH} H.R. Bennett, D. Lutzer, {\it Ordered spaces with special bases}, Fundamenta Mathematicae,
{\bf 158}(1998), 289--299.

\bibitem{BL2002} H.R. Bennett, D. Lutzer, {\it Recent developments in the topology of ordered spaces}, Recent progress in General Topology II, ed. by Hu\v{s}ek and van Mill, ELSEVIER, (2002), 85-114.

\bibitem{BT} T. Banakh, I. Guran, O. Ravsky, {\it Characterizing meager paratopological groups}, Applied General Topology,
{\bf 12(1)}(2011), 27--33.

\bibitem{CJ} J. Cao, R. Drozdowski, Z. Piotrowski, {\it Weak continuity properties of topological groups}, Czech. Math. J.,
{\bf 60(135)}(2010), 133--148.

\bibitem{E1978} R. Engelking, {\it Dimension theory}, PWN, Polish Scientific Publ., Warszawa, 1978.

\bibitem{E1989} R. Engelking, {\it General Topology} (revised and completed edition), Heldermann
Verlag, Berlin, 1989.

\bibitem{Gr} G. Gruenhage, {\it Generalized metric spaces}, In: K. Kunen, J. E.
Vaughan(Eds.), Handbook of Set-Theoretic Topology, Elsevier Science
Publishers B.V., Amsterdam, 1984, 423--501.

%\bibitem{GI} M.I. Graev, {\it Theory of topological groups}, Uspekhi Mat. Nauk,
%{\bf 5(36)}(1950), 3--56.

\bibitem{HP} K.P. Hart, J. Nagata, J.E. Vaughan, {\it Encyclopedia of General topology}, Elsevier Science
Publishers B.V., Amsterdam, 2004, 337--340.

\bibitem{LP} P.Y. Li, L. Mou, S.Z. Wang, {\it Notes on questions about spaces with algebraic strucures}, to appear in Topology Appl.

\bibitem{Ls7} S. Lin, A note on the Arens' space and sequential fan, {\it Topology Appl.}, {\bf 81}(1997), 185--196.

\bibitem{Ls3} S. Lin, Point-countable Covers and Sequence-covering Mappings(in
Chinese): Science Press, Beijing, 2002.

\bibitem{LC} C. Liu, {\it A note on paratopological groups}, Comment.Math.Univ.Carolin.,
{\bf 47}(2006), 633--640.

\bibitem{LL2010} C. Liu, S. Lin,  {\it Generalized metric spaces with algebraic structures}, {\it Topology Appl.}, {\bf 157}(2010), 1966--1974.

\bibitem{LC1} C. Liu, {\it Metrizability of paratopological (semitopological) groups}, {\it Topology Appl.}, {\bf 159}(2012), 1415--1420.

\bibitem{LD} D. Lutzer, {\it On generalized ordered spaces}, Dissertationes Math., 89 (1971).

\bibitem{MA} E. Michael, {\it A quintuple quotient quest}, {\it Gen. Topology Appl.}, {\bf 2}(1972), 91--138.

\bibitem{PN} N.M. Pyrch, A.V. Ravsky, {\it On free paratopological groups},
Matematychni Studii, {\bf 25}(2006), 115--125.

\bibitem{RO} O. Ravsky, {\it On H-closed paratopological groups}, arXiv:1003.5377v1.

\bibitem{R2001} O. Ravsky, {\it Paratopological groups I}, Mathematychni Studii, {\bf 16}(2001), 37-48.

%\bibitem{TY} Y. Tanaka, {\it Point-countable covers and $k$-networks,}, Topology Proc., {\bf 12}(1987), 327-349.

\bibitem{WJ} J.M. Worrell, H.H. Wicke, {\it Characterizations of
developable topological spaces}, {\it Canad. J. Math.}, {\bf
17}(1965), 820--830.
\end{thebibliography}
\end{document}